\documentclass[11pt]{amsart}
\usepackage{amssymb}
\usepackage{verbatim}
\usepackage{color}

\begin{document}

\newtheorem{theorem}{Theorem}    
\newtheorem{proposition}[theorem]{Proposition}
\newtheorem{conjecture}[theorem]{Conjecture}
\def\theconjecture{\unskip}
\newtheorem{corollary}[theorem]{Corollary}
\newtheorem{lemma}[theorem]{Lemma}
\newtheorem{sublemma}[theorem]{Sublemma}
\newtheorem{fact}[theorem]{Fact}
\newtheorem{observation}[theorem]{Observation}
\theoremstyle{definition}
\newtheorem{definition}{Definition}
\newtheorem{definitions}{Definitions}
\def\thedefinitions{\unskip}
\newtheorem{notation}[definition]{Notation}
\newtheorem{remark}[definition]{Remark}
\newtheorem{question}[definition]{Question}
\newtheorem{questions}[definition]{Questions}
\newtheorem{example}[definition]{Example}
\newtheorem{problem}[definition]{Problem}
\newtheorem{exercise}[definition]{Exercise}

\numberwithin{theorem}{section}
\numberwithin{definition}{section}
\numberwithin{equation}{section}

\def\reals{{\mathbb R}}
\def\torus{{\mathbb T}}
\def\heis{{\mathbb W}}
\def\integers{{\mathbb Z}}
\def\rationals{{\mathbb Q}}
\def\naturals{{\mathbb N}}
\def\complex{{\mathbb C}\/}
\def\distance{\operatorname{distance}\,}
\def\support{\operatorname{support}\,}
\def\dist{\operatorname{dist}\,}
\def\Span{\operatorname{span}\,}
\def\degree{\operatorname{degree}\,}
\def\kernel{\operatorname{kernel}\,}
\def\dim{\operatorname{dim}\,}
\def\codim{\operatorname{codim}}
\def\trace{\operatorname{trace\,}}
\def\dimension{\operatorname{dimension}\,}
\def\codimension{\operatorname{codimension}\,}
\def\nullspace{\scriptk}
\def\kernel{\operatorname{Ker}}
\def\ZZ{ {\mathbb Z} }
\def\p{\partial}
\def\rp{{ ^{-1} }}
\def\Re{\operatorname{Re\,} }
\def\Im{\operatorname{Im\,} }
\def\ov{\overline}
\def\eps{\varepsilon}
\def\lt{L^2}
\def\diver{\operatorname{div}}
\def\curl{\operatorname{curl}}
\def\etta{\eta}
\newcommand{\norm}[1]{ \|  #1 \|}
\def\expect{\mathbb E}
\def\bull{$\bullet$\ }
\def\det{\operatorname{det}}
\def\Det{\operatorname{Det}}
\def\multiR{\mathbf R}
\def\bestA{\mathbf A}
\def\Apq{\mathbf A_{p,q}}
\def\Apqr{\mathbf A_{p,q,r}}
\def\diameter{\operatorname{diameter}}
\def\bp{\mathbf p}
\def\bff{\mathbf f}
\def\bg{\mathbf g}
\def\essd{\operatorname{essential\ diameter}}

\def\mab{\max(|A|,|B|)}
\def\t2{\tfrac12}
\def\tatb{tA+(1-t)B}

\newcommand{\abr}[1]{ \langle  #1 \rangle}

\newcommand{\Norm}[1]{ \Big\|  #1 \Big\| }
\newcommand{\set}[1]{ \left\{ #1 \right\} }
\def\one{{\mathbf 1}}
\def\zero{{\mathbf 0}}
\newcommand{\modulo}[2]{[#1]_{#2}}

\def\indexset{{\mathbb J}}

\def\scriptf{{\mathcal F}}
\def\scripts{{\mathcal S}}
\def\scriptq{{\mathcal Q}}
\def\scriptv{{\mathcal V}}
\def\scriptg{{\mathcal G}}
\def\scriptm{{\mathcal M}}
\def\scriptb{{\mathcal B}}
\def\scriptc{{\mathcal C}}
\def\scriptt{{\mathcal T}}
\def\scripti{{\mathcal I}}
\def\scripte{{\mathcal E}}
\def\scriptv{{\mathcal V}}
\def\scriptw{{\mathcal W}}
\def\scriptu{{\mathcal U}}
\def\scriptS{{\mathcal S}}
\def\scripta{{\mathcal A}}
\def\scriptr{{\mathcal R}}
\def\scripto{{\mathcal O}}
\def\scripth{{\mathcal W}}
\def\scriptd{{\mathcal D}}
\def\scriptl{{\mathcal L}}
\def\scriptn{{\mathcal N}}
\def\PP{{\mathcal P}}
\def\scriptk{{\mathcal K}}
\def\scriptp{{\mathcal P}}
\def\scriptj{{\mathcal J}}
\def\scriptz{{\mathcal Z}}
\def\frakv{{\mathfrak V}}
\def\frakV{{\mathfrak V}}
\def\frakA{{\mathfrak A}}
\def\frakB{{\mathfrak B}}
\def\frakC{{\mathfrak C}}

\def\subb{\le} 
\def\Rank{\operatorname{rank}} 
\def\field{\mathbb F}
\def\Abest{{\mathbb A}}
\def\Bbest{{\mathbb B}}
\def\Abestest{{\mathbf A}}

\author{Michael Christ}
\address{
        Michael Christ\\
        Department of Mathematics\\
        University of California \\
        Berkeley, CA 94720-3840, USA}
\email{mchrist@berkeley.edu}
\thanks{Research supported in part by NSF grant DMS-0901569.}

\date{July 30, 2013.}

\title[Optimal Constants in H\"older-Brascamp-Lieb Inequalities] 
{The Optimal Constants \\ in H\"older-Brascamp-Lieb Inequalities  \\ For Discrete Abelian Groups}
\maketitle

\section{Introduction}

An Abelian group HBL datum is a tuple
\[\scriptg=(G,(G_j: 1\le j\le m),(\phi_j: 1\le j\le m))\]
where $m$ is a positive integer, 
$G$ and $G_j$ are Abelian groups, and $\phi_j: G\to G_j$ are group homomorphisms. 
A finitely generated Abelian group HBL datum $\scriptg$
is one for which all of the groups $G,G_j$ are finitely generated;
a finite Abelian group HBL datum is one for which all of these groups are finite.

In this paper we determine the optimal constants $A\in[0,\infty]$ for multilinear inequalities
\begin{equation} \label{maininequality} 
\sum_{x\in G} \prod_j f_j(\phi_j(x)) \le A\prod_j\norm{f_j}_{1/s_j} \end{equation}
associated to arbitrary Abelian group HBL data.
Here the exponents $s_j$ belong to $[0,1]$, and 
$f_j:G_j\to[0,\infty)$ are arbitrary nonnegative functions in $L^{1/s_j}(G_j)$.
The underlying measure on $G_j$ is counting measure.

There is a substantial literature concerning corresponding inequalities in the continuum setting,
with $G,G_j$ replaced by $\reals^d,\reals^{d_j}$, and using Lebesgue measure in place of counting measure
to define the $L^{1/s_j}$ norms. See for instance 
\cite{bartheicm} and its bibliography.  In that continuum setting, a necessary and sufficient
condition for there to exist a finite constant for which \eqref{maininequality} holds 
was established in \cite{BCCT08} and in \cite{BCCT10}, by two different arguments.
The case in which all $d_j$ equal $1$ was treated earlier \cite{carlenliebloss}.
Still earlier \cite{liebgaussian}, it was shown that the supremum over arbitrary nonnegative
functions with specified norms equals the supremum over the subclass of all real Gaussian functions with those norms,
whether or not this supremum is finite.

Inequalities for discrete Abelian groups were considered in \cite{BCCT10}, where a necessary
and sufficient condition was given for there to exist a finite constant for which 
\eqref{maininequality} holds. In \cite{CDKSY} an application of such inequalities to computer science
was developed, and it was shown that if $G$ is torsion-free, then the optimal constant in the inequality 
equals $1$ in all cases in which it is finite.  

For finite groups $G$, for every $s$, the inequality holds with some finite constant. The main thrust of this
paper is the determination of those constants. The results for finite groups and for torsion-free groups are then easily
combined to yield the general case.

\subsection{Notations}
The number of elements of a set $E$ is denoted by $|E|$.
All $L^p$ norms in this paper are defined with respect to counting measure; 
\begin{equation*} \norm{f}_{L^p(G)} = \big( \sum_{x\in G} |f(x)|^p\big)^{1/p}\end{equation*} 
for $p\in[1,\infty)$ while $\norm{f}_{L^\infty(G)} = \sup_{x\in G}|f(x)|$.
Thus all Abelian groups are implicitly regarded as being discrete. 

We will often simplify notation by denoting $\scriptg$ instead by $(\phi_j: G\to G_j: 1\le j\le m)$ or, 
more simply, $(\phi_j:G\to G_j)$.  
$\zero$ will denote the group $\set{0}$, which will be considered to be a subgroup of all other groups.
The notation $H\le G$ signifies that $H\subset G$ is a subgroup, while $H<G$ means that $H$ is a proper subgroup.

\subsection{Key definitions}
\begin{definition}
Let $\scriptg=(\phi_j: G\to G_j)$ be an Abelian group HBL datum.
Let $s\in[0,1]^m$.
For any finite subgroup $H\le G$, 
\begin{align} A(H,s) = |H|\prod_j |\phi_j(H)|^{-s_j}.  \end{align}
\end{definition}

\begin{definition}
Let $\scriptg=(\phi_j:G\to G_j)$ be any Abelian group HBL datum.  For $s\in[0,1]^m$, 
\begin{equation} \Abest(\scriptg,s) = \sup_{H\le G} A(H,s) \end{equation} 
where the supremum is taken over all finite subgroups $H$ of $G$.
\end{definition}


\begin{definition}
For any Abelian group HBL datum $\scriptg$  and any $s\in[0,1]^m$,
$\Bbest(\scriptg,s)\in[0,\infty]$ is the infimum of the set of all $C\le\infty$ such that
\begin{equation} \sum_{x\in G} \prod_j f_j(\phi_j(x)) \le \Bbest(\scriptg,s) \prod_j\norm{f_j}_{1/s_j} \end{equation}
for all nonnegative functions $f_j\in L^{1/s_j}(G_j)$.
\end{definition}

\begin{definition}
Let $\scriptg$ be an Abelian group HBL datum.  $\scriptp(\scriptg)$ is the set of all $s\in[0,1]^m$ such that
\begin{equation*} \Rank(H)\le \sum_{j=1}^m s_j\Rank(\phi_j(H)) \ \text{ for every subgroup $H\le G$ of finite rank.} 
\end{equation*} \end{definition}

$\scriptp(\scriptg)$ and $\Abest(\scriptg,s)$ measure complementary algebraic aspects of $\scriptg$.
These are the structural quantities in terms of which the optimal constants $\Bbest(\scriptg,s)$ of the
analytic inequalities are expressed.

\begin{definition}
If $\scriptg=(\phi_j:G\to G_j: 1\le j\le m)$ 
and $\scriptg'=(\phi'_j:G'\to G'_j: 1\le j\le m)$ are Abelian group HBL data,
$\scriptg'$ is a sub-datum of $\scriptg$ if $G'\le G$, $G'_j\le G_j$,
and $\phi'_j$ is the restriction of $\phi_j$ to $G'$.
\end{definition}

\subsection{Finite groups}

Let $\scriptg=(\phi_j: G\to G_j: 1\le j\le m)$ be a finite Abelian  group HBL datum.
For any exponents $s_j\in[0,1]$ there exists $A<\infty$, depending also on $\scriptg$, 
for which the inequality \eqref{maininequality} is valid. That is, $\Bbest(\scriptg,s)<\infty$
for all finite Abelian group data and all exponents.

\begin{theorem} \label{thm:finitegroups}
For all finite Abelian group data $\scriptg$ and all $s\in[0,1]^m$,
\begin{equation} \label{BleA} \Bbest(\scriptg,s)= \Abest(\scriptg,s).  \end{equation} \end{theorem}

That is, the optimal constant in the associated inequality \eqref{maininequality} is $\Abest(\scriptg,s)$.

\subsection{General groups}

\begin{theorem}\label{thm:generalgroups}
For any Abelian group HBL datum
$\scriptg=(\phi_j: G\to G_j: 1\le j\le m)$ and any $s\in\scriptp(\scriptg)$,
\begin{equation} \label{generalgroupsmaininequality}
\Bbest(\scriptg,s)= \Abest(\scriptg,s).  \end{equation} 
Conversely, if $s\in[0,1]^m$ and if there exists $C<\infty$ such that
\begin{equation} \label{setcase} |E|\le C\prod_j |\phi_j(E)|^{s_j} \end{equation}
for all nonempty finite subsets $E\subset G$, then $s\in\scriptp(\scriptg)$ and $C\ge\Abest(\scriptg,s)$.
\end{theorem}

For general data $\scriptg$, $\Abest(\scriptg,s)$ can be infinite. 
If so, then \eqref{generalgroupsmaininequality} is not valid with any finite constant, even if $s\in\scriptp(\scriptg)$. 
If there exists a finite subgroup $H\le G$ which satisfies $A(H,s)=\Abest(\scriptg,s)$, 
then equality is realized in \eqref{setcase} by $E=H$,
and equality is realized in \eqref{generalgroupsmaininequality} with 
each function $f_j$ equal to the indicator function of $\phi_j(H)$.

This result synthesizes Theorem~\ref{thm:finitegroups} with the result for torsion-free finitely generated
$G$ obtained in \cite{CDKSY}, extends this synthesis to groups which are not finitely generated.

\section{Preliminary facts}

Obviously $\Abest(\scriptg,s)\ge 1$ for all data and exponents, since $A(\zero,s)=1$.
If $|G|=1$ then $\Abest(s)=1$, and \eqref{generalgroupsmaininequality} certainly holds.

It is always the case that
\begin{equation*} \Bbest(\scriptg,s)\le |G|.\end{equation*}
Indeed, 
\begin{equation*}
\sum_{x\in G}\prod_j f_j(\phi_j(x)) \le |G|\prod_j\norm{f_j}_\infty
\le |G|\prod_j \norm{f_j}_{1/s_j}
\end{equation*}
for all nonnegative functions. 

\begin{lemma} Let $\scriptg$ be any Abelian group HBL datum and let $H\le G$ be any finite subgroup.
Then $\Bbest(\scriptg,s)\ge A(H,s)$.
\end{lemma}

\begin{proof}
Let $H\le G$.  It suffices to apply the inequality to the functions 
\begin{equation*} f_j = \one_{\phi_j(H)} 
= \begin{cases} 1 &\text{if } x\in H \\ 0 &\text{if } x\notin H.\end{cases}
\end{equation*}
\end{proof}

Taking the supremum over all such subgroups $H$ gives
\begin{corollary}
For any Abelian group HBL datum $\scriptg$ and any $s\in[0,1]^m$,
\begin{equation} \Bbest(\scriptg,s)\ge \Abest(\scriptg,s).\end{equation}
\end{corollary}

\begin{lemma} \label{lemma:injectivecase}
Let $K$ be the intersection of the kernels of $\phi_k$ for all $k$ such that $s_k=1$.
If $K=\zero$ then 
\begin{equation*} \Bbest(\scriptg,s)=1. \end{equation*}
\end{lemma}

In particular, in this circumstance $\Bbest(\scriptg,s)\le\Abest(\scriptg,s)$
since $\Abest(\scriptg,s)\ge A(\zero,s)= 1$ for arbitrary $\scriptg,s$. 

\begin{proof}
Since the mapping $G\owns x\mapsto (\phi_k(x): s_k=1)$ is injective,
for arbitrary functions $f_j:G_j\to[0,\infty)$,
\begin{equation*} \prod_{k:\,s_k=1} \norm{f_k}_{L^1(\phi_k(G)}
\le \sum_{x\in G} \prod_{k:\,s_k=1} f_k(\phi_k(x)).  \end{equation*}
Therefore
\begin{align*}
\sum_{x\in G} \prod_j f_j(\phi_j(x))
&\le
\prod_{j:\,\, s_j\ne 1} \norm{f_j}_\infty
\sum_{x\in G}\,\, \prod_{k:\,s_k=1} f_k(\phi_k(x))
\\
&\le
\prod_{j:\,\, s_j\ne 1} \norm{f_j}_\infty\,\,
\prod_{k:\,s_k=1} \norm{f_k}_{L^1(\phi_k(G))}
\\
&\le \prod_{n=1}^m \norm{f_n}_{1/s_n}.
\end{align*}
\end{proof}

\begin{lemma} \label{lemma:icase}
If there exists an index $i$ such that $s_j=0$ for all $j\ne i$ then
\begin{equation*} \Abest(\scriptg,s) = \Bbest(\scriptg,s)  =|G|^{1-s_i}|\kernel(\phi_i)|^{s_i} \end{equation*}
\end{lemma}


\begin{proof}
Let $i$ be an index such that $s_j=0$ for all $j\ne i$.
Set $\scriptk_i=\kernel(\phi_i)$.
For any subgroup $H\le G$, 
\begin{align*}
A(H,s) = |H|\cdot|\phi_i(H)|^{-s_i}
= |H|\cdot (|H|/|H\cap\scriptk_i|)^{-s_i}
= |H|^{1-s_i}  |H\cap\scriptk_i|^{s_i}
\le  |G|^{1-s_i}  |\scriptk_i|^{s_i}.
\end{align*}
Thus in this situation,
\begin{equation} \label{identifyinicase}
\Abest(\scriptg,s) = A(G,s)=|G|^{1-s_i}|\scriptk_i|^{s_i}.
\end{equation}

To analyze $\Bbest(\scriptg,s)$,
let $Y$ be a set that contains exactly one element from each coset in $G/\scriptk_i$.
Then
\begin{align*}
\sum_{x\in G} \prod_j f_j(\phi_j(x))
\le \prod_{j\ne i}\norm{f_j}_\infty \sum_{x\in G} f_i(\phi(x)),
\end{align*}
and
\begin{align*}
\sum_{x\in G} f_i(\phi_i(x))
&= \sum_{y\in Y} \sum_{z\in\scriptk_i} f_i(\phi_i(z)+\phi_i(y)).
\\&= \sum_{y\in Y} \sum_{z\in\scriptk_i} f_i(\phi_i(y))
\\&= |\scriptk_i| \sum_{y\in Y}  f_i(\phi_i(y)).
\end{align*}
Since the images $\phi_i(y)$ are distinct for distinct values of $y$,
this implies that
\begin{align*}
\sum_{x\in G} f_i(\phi_i(x))
&\le |\scriptk_i| \norm{f_i}_1
\\&\le |\scriptk_i| \norm{f_i}_{1/s_i} |\phi_i(Y)|^{1-s_i}
\\&= |\scriptk_i| \norm{f_i}_{1/s_i} |\phi_i(G)|^{1-s_i}
\\&= |\scriptk_i| \norm{f_i}_{1/s_i} (|G|/|\scriptk_i|)^{1-s_i}
\\&= |\scriptk_i|^{s_i}|G|^{1-s_i} \norm{f_i}_{1/s_i}.
\\&= \Abest(\scriptg,s) \norm{f_i}_{1/s_i}
\end{align*}
using \eqref{identifyinicase}.
Thus $\Bbest(\scriptg,s)\le \Abest(\scriptg,s)$.
It has already been noted that the converse inequality holds, as a direct consequence of the definitions.
\end{proof}

\begin{lemma}\label{lemma:smallK}
For any Abelian group HBL datum $\scriptg$ and any $s\in[0,1]^m$,
\begin{equation} \big|\bigcap_{j=1}^m \kernel(\phi_j)\big|\le \Abest(\scriptg,s). \end{equation}
\end{lemma}

\begin{proof}
The subgroup $\scriptk=\cap_j \kernel(\phi_j)$ satisfies
\[ |\scriptk|= A(\scriptk,s)\prod_j |\phi_j(\scriptk)|^{s_j} = A(\scriptk,s)\prod_j 1 = A(\scriptk,s)\le\Abest(\scriptg,s).\]
\end{proof}

\begin{lemma}\label{lemma:subdata}
If $\tilde\scriptg$ is a sub-datum of $\scriptg$ and if $s\in[0,1]^m$ then $\Abest(\tilde\scriptg)\le\Abest(\scriptg)$.
\end{lemma}

This is an immediate consequence of the definition of $\Abest(\cdot,s)$ since any subgroup of $\tilde G$ is a subgroup of $G$.
\qed

\section{Factorization}


Let $\scriptg=(\phi_j:G\to G_j)$ be an HBL datum. 
To any subgroup $G'\le G$  is associated the HBL sub-datum
$\scriptg' = (\phi_j:G'\to \phi_j(G'))$,
where the restriction of $\phi_j$ to $G'$ is denoted again by $\phi_j$.
There is also an associated HBL quotient datum,
\begin{equation*} \scriptg/\scriptg'
=(G/G',\set{G_j/\phi_j(G')},\set{\psi_j})\end{equation*}
where $\psi_j:G/G'\to G_j/\phi_j(G')$ is the homomorphism associated to $\phi_j$ via the quotient mappings
$G\to G/G'$ and $G_j\to G_j/\phi_j(G')$.

\begin{lemma} \label{lemma:Bfactorize} Let $\scriptg$ be an HBL datum, and let $s\in[0,1]^m$.
Let $G'\le G$ be any subgroup, and let $\scriptg'$ be the HBL sub-datum associated to $G'$.  Then 
\begin{equation*} \Bbest(\scriptg,s)\le \Bbest(\scriptg/\scriptg',s) \cdot\Bbest(\scriptg',s).\end{equation*}
\end{lemma}
This lemma is valid for arbitrary Abelian group HBL data, with no hypothesis of finiteness or finite generation. 

\begin{proof}
Let $G'_j=\phi_j(G')$ and 
let $\psi_j:G/G'\to G_j/G'_j$ be the homomorphisms associated to $\phi_j$ via the quotient maps.
Let $f_j\in L^{1/s_j}(G_j)$ be nonnegative functions with finite norms.
We may suppose without loss of generality that $s_j>0$ for all $j$,
by majorizing $f_j(x)$ by $\norm{f_j}_{L^\infty}$ for all $j$ for which $s_j=0$, and then dropping those indices.

For $x+G'_j\in G_j/G'_j$ define 
\[F_j(x+G'_j) = \big(\sum_{y\in G'_j} |f_j(x+y)|^{1/s_j}\big)^{s_j}.\]
Then \[\norm{F_j}_{L^{1/s_j}(G_j/G'_j)} = \norm{f_j}_{L^{1/s_j}(G_j)}.\]
By definition of $\Bbest(\scriptg/\scriptg',s)$,
\[\sum_{x\in G/G'} \prod_j F_j(\psi_j(x)) 
\le \Bbest(\scriptg/\scriptg',s)\prod_j \norm{F_j}_{L^{1/s_j}(\psi_j(G_j/G'_j))}.\]

Let $T\subset G$ be a subset which contains exactly one element from each coset $x+G'$. 
Then for any nonnegative functions $f_j$,
\begin{align*}
\sum_{x\in G} \prod_j f_j(\phi_j(x))
&= \sum_{t\in T} \sum_{y\in G'} \prod_j f_j(\phi_j(t+y)) 
\\&= \sum_{t\in T} \sum_{y\in G'} \prod_j f_j(\phi_j(t) + \phi_j(t)) 
\\&\le  \sum_{t\in T} \Bbest(\scriptg',s) \prod_j F_j(\phi_j(t)+G'_j)
\\&\le \Bbest(\scriptg/\scriptg',s) \cdot\Bbest(\scriptg',s)  \prod_j \norm{F_j}_{1/s_j}
\\&=   \Bbest(\scriptg/\scriptg',s) \cdot\Bbest(\scriptg',s)\prod_j \norm{f}_{1/s_j}.
\end{align*}
\end{proof}

\begin{lemma} \label{lemma:Afactorize}
Let $\scriptg$ be a finite Abelian group HBL datum, and let $s\in[0,1]^m$. Let $G'\le G$ be a subgroup,
and let $\scriptg'$ be the associated HBL sub-datum.  If $A(G',s) = \Abest(\scriptg',s)$ then
\begin{equation} \label{Afactorization}
\Abest(\scriptg/\scriptg',s)\Abest(\scriptg',s) \le \Abest(\scriptg,s).  \end{equation}
\end{lemma}

\begin{proof}
Choose a subgroup $\Gamma\le G/G'$ satisfying 
$A(\Gamma,s) = \Abest(\scriptg/\scriptg',s)$.
Choose a subgroup $H'\le G$ whose image under the quotient map from $G$ to $G/G'$
equals $\Gamma$. For each index $j$, 
\[|\phi_j(H')| = |\psi_j(\Gamma)|\cdot |\phi_j(G')|\]
by elementary group theory.
Consequently
\begin{align*}
A(H',s) &=|H'|\prod_j |\phi_j(H')|^{-s_j}
\\&= |H'|\prod_j |\psi_j(\Gamma)|^{-s_j} |\phi_j(G')|^{-s_j}
\\&= |\Gamma|\prod_j |\psi_j(\Gamma)|^{-s_j}\cdot |G'|\prod_j |\phi_j(G')|^{-s_j}.
\\&= A(\Gamma,s)A(G',s)
\\&= \Abest(\scriptg/\scriptg',s)\cdot\Abest(\scriptg',s).
\end{align*}
Since $\Abest(\scriptg,s)\ge A(H',s)$ by definition,
this establishes \eqref{Afactorization}.
\end{proof}


Factorization is a fundamental tool in \cite{BCCT10}. In that work, one factors with respect
to vector subspaces or subgroups which have an extremal property called criticality. The hypothesis
that $A(G',s)=\Abest(\scriptg',s)$ in Lemma~\ref{lemma:Afactorize} is an analogue of criticality.
Lemma~\ref{lemma:Bfactorize}, in contradistinction, has no corresponding hypothesis, and yields an inequality
in the opposite direction.

\section{Equality in a fundamental case} 

The special case in which all exponents belong to $\set{0,1}$, with at most one 
exception, is fundamental to the analysis.
\begin{lemma} \label{lemma:icase2}
Let $\scriptg$ be a finite Abelian group HBL datum, and let $s\in[0,1]^m$.
If $s_j\in\set{0,1}$ for all but at most one index $j$ then
\begin{equation} \Abest(\scriptg,s) = \Bbest(\scriptg,s).\end{equation}
\end{lemma}

\begin{proof}
Choose an index $i$ such that $s_j\in\set{0,1}$ for all $j\ne i$.
Let $G'$ be the intersection of $\kernel(\phi_k)$, taken over all $k\ne i$ satisfying $s_k=1$.
Let $\scriptg'$ be the finite Abelian group HBL datum $\scriptg'=(\phi_j|_{G'}: G'\to\phi_j(G'))$.
Let $\scriptg/\scriptg'$ be the quotient datum.

Define $t$ by $t_i=s_i$, and $t_j=0$ for all $j\ne i$.
Then $\Abest(\scriptg',t)=\Abest(\scriptg',s)$. Indeed, for each $j\ne i$, for each subgroup $H\le G'$,
$|\phi_j(H)|^{-s_j}=|\phi_j(H)|^{-t_j}$. If $j=i$ of if $s_j=0$ this holds because $t_j=s_j$.
If $s_j=1$ then $|\phi_j(H)|=1$ since $H\le G'\le\kernel(\phi_j)$, so again $|\phi_j(H)|^{-s_j}=|\phi_j(H)|^{-t_j}$.

Likewise $\Bbest(\scriptg',t) = \Bbest(\scriptg',s)$. 
Indeed, letting $K$ be the set of all indices $k\ne i$ such that $s_k=1$,
$\phi_k(x)=0$ for all $x\in G$ and all $k\in K$ and consequently
\begin{equation*}
\sum_{x\in G'} \prod_j f_j(\phi_j(x)) = \prod_{k\in K}f_k(\phi_k(0)) \prod_{j\notin K} f_j(\phi_j(x)).
\end{equation*}
Thus $\Bbest((\phi_j: G'\to G_j),s) = \Bbest((\phi_j: G'\to G_j: j\notin K),s)$
and likewise $\Bbest((\phi_j: G'\to G_j),t) = \Bbest((\phi_j: G'\to G_j: j\notin K),t)$.
Since $s_j=t_j$ for all $j\notin K$, it follows that $\Bbest(\scriptg',t) = \Bbest(\scriptg',s)$. 

$\Abest(\scriptg',t)=\Bbest(\scriptg',t)$ by Lemma~\ref{lemma:icase}.
Therefore $\Abest(\scriptg',s)=\Bbest(\scriptg',s)$.

We also conclude from Lemma~\ref{lemma:icase} that 
\begin{equation}\label{needthistofactor} 
\Abest(\scriptg',s) = |G'|^{1-s_i}|\kernel(\phi_i|_G')|^{s_i} = A(G',s).\end{equation}

Now consider $\scriptg/\scriptg'$. Since $G'=\cap_{k\in K} \kernel(\phi_k)$,
the intersection over all $k\in K$ of the kernels of the quotient mappings $[\phi_k]: G/G'\to G_k/\phi_k(G')$
is $\set{0}$. Therefore 
\begin{equation*} A(H,s) = |H|\prod_{j}|[\phi_j](H)|^{-s_j} \le  |H|\prod_{k\in K}|[\phi_k](H)|^{-s_k} \le 1\end{equation*}
for all $H\le G/G'$.  Thus $\Abest(\scriptg/\scriptg/,s)=1$.

Define $t_i=0$ and $t_j=s_j$ for all $j\ne i$.
According to Lemma~\ref{lemma:Bfactorize}, $\Bbest(\scriptg/\scriptg',t)=1$.
Since $s_j\ge t_j$ for all $j$, $\Bbest(\scriptg/\scriptg',s) \le \Bbest(\scriptg/\scriptg',t)=1$.
Since $\Bbest(\scriptg/\scriptg',s)\ge 1$ for arbitrary data, we conclude that $\Bbest(\scriptg/\scriptg',s)= 1$.
So $\Bbest(\scriptg/\scriptg',s) = \Abest(\scriptg/\scriptg',s)$.

Since $G'$ satisfies $A(G',s)=\Abest(\scriptg',s)$ by \eqref{needthistofactor},
Lemma~\ref{lemma:Afactorize} can be applied to conclude that 
$\Abest(\scriptg,s) \ge \Abest(\scriptg',s)\Abest(\scriptg/\scriptg',s)$.
Therefore an invocation of first Lemma~\ref{lemma:Bfactorize}, then the equalities shown above,
then Lemma~\ref{lemma:Afactorize} yields
\begin{equation*} \scriptb(\scriptg,s) \le \scriptb(\scriptg',s)\scriptb(\scriptg/\scriptg',s)
= \scripta(\scriptg',s)\scripta(\scriptg/\scriptg',s) \le \scripta(\scriptg,s).  \end{equation*}
Since the converse inequality holds for all HBL data, we conclude that
$\scriptb(\scriptg,s) = \scripta(\scriptg,s)$. 
\end{proof}

\section{Conclusion of proof for finite groups}

Let $\scriptg=(\phi_j:G\to G_j)$ be a finite Abelian group HBL datum.
To complete the proof of Theorem~\ref{thm:finitegroups} we argue by induction on the cardinality of $G$.

\begin{lemma} \label{lemma:splitting} 
Let $\scriptg$ be a finite Abelian group HBL datum, and let $s\in[0,1]^m$. 
Let $0<G'<G$ be a subgroup of $G$, and let $\scriptg'$ be the associated Abelian group HBL sub-datum.  If 
\begin{equation}  \label{splittingcase} A(G',s) = \Abest(\scriptg',s) \end{equation}
then $\Bbest(\scriptg,s)=\Abest(\scriptg,s)$.
\end{lemma}

\begin{proof}
Consider $\scriptg' = (\phi_j:G'\to \phi_j(G'))$,
where the restriction of $\phi_j$ to $G'$ is denoted again by $\phi_j$.
By induction on $|G|$, $\Abest(\scriptg',s) = \Bbest(\scriptg',s)$
and $\Abest(\scriptg/\scriptg',s)=\Bbest(\scriptg/\scriptg',s)$.
By the preceding lemmas,
\begin{align*}
\Bbest(\scriptg,s) 
\le \Bbest(\scriptg/\scriptg',s)\Bbest(\scriptg',s)
= \Abest(\scriptg/\scriptg',s)\Abest(\scriptg',s)
\le\Abest(\scriptg,s).
\end{align*}
On the other hand, it has already been noted that $\Abest(\scriptg,s)\le\Bbest(\scriptg,s)$,
as a direct consequence of their definitions.
\end{proof}

\begin{lemma} \label{lemma:secondsplitting} 
Let $\scriptg$ be a finite Abelian group HBL datum, and let $s\in[0,1]^m$. 
If there exists a subgroup $\zero<H<G$ satisfying $A(H,s)\ge 1$
then $\Bbest(\scriptg,s)=\Abest(\scriptg,s)$.  \end{lemma}

\begin{proof}
In this case, there exists a subgroup $\zero<G'<G$ satisfying \eqref{splittingcase}.
Indeed, the set $\scripts$
of all subgroups $\zero<H<G$ 
for which $A(H,s)$ is maximal among all such subgroups,
is a partially ordered set under inclusion, and it is given that this
set is nonempty. It is finite since $G$ is a finite group, and therefore has at least one minimal element.

Choose $G'$ to be such a minimal element. Then $G'$ is a proper nonzero subgroup  of $G$
which satisfies \eqref{splittingcase}.  By Lemma~\ref{lemma:splitting}, $\Bbest(\scriptg,s)=\Abest(\scriptg,s)$.
\end{proof}

We may assume that $|G|>1$, since the conclusion holds when $G=\zero$.
Two cases remain to be treated. In the first of these cases,
$A(G,s)=\Abest(\scriptg,s)$ and $A(G',s)<\Abest(\scriptg,s)$ for all $\zero<G'<G$.
In the second, $A(G',s)<1$ for every nonzero subgroup $G'\le G$, including $G$ itself. 
Consequently $\Abest(\scriptg,s)=1$.

Consider the second case.  Let a parameter $\theta$ vary over $[0,1]$.
$\Abest(\scriptg,\theta s)$ is a continuous nonincreasing function of $\theta$.
For $\theta=1$, it is given that $A(H,\theta s)<1$ for every subgroup $0<H\le G$.
For $\theta=0$, $A(G,\theta s)=|G|$. 
Since $A(G,0s)=|G|>1=\Abest(\scriptg,s)$, there exists a smallest $\theta^\star\in(0,1)$
for which $\Abest(\scriptg,\theta^\star s)=1$ and $A(H,\theta^\star s)=1$
for some subgroup satisfying $\zero<H\le G$.
Since $\Bbest(\scriptg,s)\le\Bbest(\scriptg,\theta^\star s)$, 
it suffices to prove
that $\Bbest(\scriptg,\theta^\star s)\le \Abest(\scriptg,\theta^\star s)$,
since the latter is equal to $1=\Abest(\scriptg,s)$.
If $H<G$ then Lemma~\ref{lemma:secondsplitting} applies and gives the desired conclusion.
If $H=G$ then matters have been reduced to the first of the two cases described above.
So it suffices to treat that first case.

For each $\Abestest\in[1,\infty)$ define \begin{equation} \scriptp_\Abestest
=\set{ s\in[0,1]^m: \Abest(\scriptg,s)\le\Abestest.}.\end{equation}
$\scriptp_\Abestest$ is a closed convex polytope. 
An equivalent way to state the inequality \eqref{BleA}
is that for each $\Abestest\in[1,\infty)$, 
$\Bbest(\scriptg,t)\le\Abestest$ for all $t\in\scriptp_{\Abestest}$.
Moreover, it suffices to prove this for those $\Abestest$ for which
there exists at least one $s\in[0,1]^m$ for which $\Abest(\scriptg,s)=\Abestest$.
By complex interpolation, it suffices to prove 
that $\Bbest(\scriptg,t)\le \Abestest$ for each extreme point $t$ of $\scriptp_\Abestest$.
In view of the various reductions made above, it suffices to analyze
those extreme points which fall into the first case described above.

\begin{lemma} \label{lemma:extremepts}
Let $\scriptg$ be a finite Abelian group HBL datum. 
Let $t=(t_j: 1\le j\le m)\in[0,1]^m$ be an extreme point of $\scriptp_\Abestest$
such that $A(G,t)\ge 1$, and $A(G',t)<A(G,t)$ for every subgroup $\zero<G'<G$.
Then either there exists a subgroup $\zero<H<G$ for which $A(H,t)=\Abestest$,
or $t_j\in\set{0,1}$ for all but at most one index $j$.
\end{lemma}
These possibilities are not mutually exclusive.

\begin{proof}
If $A(G,t)<\Abestest$ then $A(H,t)\le A(G,t)<\Abestest$ for every
subgroup $H\le G$. Now $t$ cannot be an extreme point unless $t_j\in\set{0,1}$ for
every index $j$, for otherwise it would be possible to freely vary at least one
coordinate $t_j$ in either direction up to some small threshold without leaving $\scriptp_\Abestest$,
contradicting extremality.

Otherwise $A(G,t)=\Abestest$ and $A(H,t)<\Abestest$ whenever $\zero<H<G$.
Then in some sufficiently small neighborhood of $t$, $\scriptp_\Abestest$ coincides with the
set of all $s\in[0,1]^m$ satisfying 
\begin{equation} \label{thatsall} \ln|G|-\sum_j s_j\ln|\phi_j(G)|  
\le \ln\Abestest;\end{equation}
the subgroup $\zero$ imposes no constraint on $s$, and constraints imposed by
all other proper subgroups are satisfied for all $s$ sufficiently close to $t$,
by continuity, since they hold with strict inequality for $s=t$.

If there were to exist two indices $i,j\in\set{1,2,\cdots,m}$
such that neither $t_i,t_j$ belonged to $\set{0,1}$,
then $t$ would lie in the interior of a line segment of points satisfying \eqref{thatsall}.
This segment would be contained in $\scriptp_\Abestest$ in some neighborhood of $t$,
again contradicting the assumed extremality of $t$.
\end{proof}

We are now in a position to complete the proof that for any finite Abelian group HBL datum
$\scriptg$ and every $\Abestest\in[1,\infty)$, $\Bbest(\scriptg,t)\le\Abestest$
for every extreme point $t$ of $\scriptp_{\Abestest}$. 
As shown above, it suffices to treat the case in which $t_j\in\set{0,1}$ for all but at most one index $j$.
In that case, $\Bbest(\scriptg,t)=\Abest(\scriptg,t)\le\Abestest$ by Lemma~\ref{lemma:icase}.

This in turn completes the proof of Theorem~\ref{thm:finitegroups}. 
\qed

  
\section{Extension to general groups} \label{section:infinitegroups}

\begin{proof}[Proof of Theorem~\ref{thm:generalgroups} for finitely generated groups]
Let $\scriptg=(\phi_j:G\to G_j)$ be an Abelian group HBL datum, and let $s\in[0,1]^m$.
Assume that $G$ is finitely generated.

Let $T\le G$ be the torsion subgroup of $T$, which is a finite group since $G$ is finitely generated.
Let $\scriptt=(\phi_j|_T: T\to \phi_j(T))$ be the associated finite Abelian group HBL datum.
By Theorem~\ref{thm:finitegroups}, $\Bbest(\scriptt,s)=\Abest(\scriptt,s)$.

Consider also the quotient group $G/T$ and the quotient quotient HBL datum $\scriptg/\scriptt$. 
Since $G/T$ is torsion-free, $\Abest(\scriptg/\scriptt,s)=1$. 
By \cite{CDKSY}, $\Bbest(\scriptg/\scriptt,s)=1$ also.

Since every finite subgroup of $G$ is contained in $T$, $\Abest(\scriptt,s) = \Abest(\scriptg,s)$.
By Lemma~\ref{lemma:Bfactorize},
\begin{equation*} \Bbest(\scriptg,s)\le\Bbest(\scriptt,s)\Bbest(\scriptg/\scriptt,s)
= \Abest(\scriptt,s)\cdot 1 = \Abest(\scriptg,s).  \end{equation*}
\end{proof}

\begin{proof}[Proof of Theorem~\ref{thm:generalgroups} in the general case]
Let $\scriptg$ be an Abelian group HBL datum, and let $s\in\scriptp(G)$.
It suffices to prove that
\begin{equation}\label{oncemore} 
\sum_{x\in G} \prod_j f_j(\phi_j(x)) \le \Abest(\scriptg,s) \prod_j\norm{f_j}_{L^{1/s_j}}\end{equation}
for arbitrary nonnegative functions $f_j$ having finite supports, for the same inequality
for general nonnegative $f_j\in L^{1/s_j}(G_j)$ follows from this special case via the Monotone Convergence Theorem.
It suffices to prove this under the hypothesis that $\Abest(\scriptg,s)<\infty$.

Denote by $\prod_j G_j$ the Cartesian product of the sets $G_j$.
By Lemma~\ref{lemma:smallK}, $|\cap_j\kernel(\phi_j)|\le \Abest(\scriptg,s)<\infty$.
Therefore the mapping $\Phi:G\to\prod_j G_j$ defined by
$\Phi(x)=(\phi_j(x): 1\le j\le m)$ is $K$-to-one where $K\le\Abest(\scriptg,s)<\infty$. 

Let $f_j$ be nonnegative and have finite supports. Then the set of all $x\in G$ for which $\prod_j f_j(\phi_j(x))\ne 0$
is the inverse image under $\Phi$ of a finite product of finite sets, so is a finite set.
Let $\tilde G\le G$ be the subgroup of $G$ generated by this set. Consider the
finitely generated Abelian group HBL datum 
$\tilde\scriptg=(\phi_j|_{\tilde G}: \tilde G\to \phi_j(\tilde G): 1\le j\le m)$.
We have already shown that $\Bbest(\tilde\scriptg)=\Abest(\tilde\scriptg)$.
Since $\tilde\scriptg$ is a sub-datum of $\scriptg$, $\Abest(\tilde\scriptg)\le\Abest(\scriptg)$
by Lemma~\ref{lemma:subdata}.  Putting this all together,
\begin{align*}
\sum_{x\in G}\prod_j f_j(\phi_j(x))
&=\sum_{x\in \tilde G}\prod_j f_j(\phi_j(x))
\\&\le \Bbest(\tilde\scriptg,s)\prod_j\norm{f_j}_{L^{1/s_j}}
\\&= \Abest(\tilde\scriptg,s)\prod_j\norm{f_j}_{L^{1/s_j}}
\\&\le \Abest(\scriptg,s)\prod_j\norm{f_j}_{L^{1/s_j}},
\end{align*}
as was to be proved.
\end{proof}

So far we have only proved the first conclusion of Theorem~\ref{thm:generalgroups}. The second, converse, portion
is much simpler.
It is given that $|E|\le C\prod_j|\phi_j(E)|^{s_j}$ for all finite sets $E\subset G$, and that $C<\infty$.
If $H\le G$ is any finite subgroup, applying this with $E=H$ gives $A(H,s)\le C$ and therefore $C\ge \Abest(\scriptg,s)$.

Consider any finitely generated subgroup $\tilde G\subset G$, and let $\tilde\scriptg$ be the associated sub-datum. 
It was proved in \cite{CDKSY} that for finitely generated Abelian group HBL data $\tilde\scriptg$, the hypothesis
that $|E|\le C\prod_j|\phi(E)|^{s_j}$ for all finite subsets $E\subset\tilde G$ implies that $s\in\scriptp(\tilde\scriptg)$.
Therefore $\Rank(\tilde G)\le \sum_j s_j\Rank(\phi_j(\tilde G))$. The validity of this inequality for every
finitely generated subroup of $G$ is the criterion for $s$ to be an element of $\scriptp(\scriptg)$. \qed


\begin{thebibliography}{20}

\bibitem{bartheicm}
F.~Barthe,
{\em The Brunn-Minkowski theorem and related geometric and functional inequalities}, 
International Congress of Mathematicians. Vol. II, 1529--1546, Eur. Math. Soc., Z\"urich, 2006

\bibitem{BCCT08}
J.~Bennett, A.~Carbery, M.~Christ, and T.~Tao,
{\em The Brascamp-Lieb inequalities: finiteness, structure and extremals}, 
Geom. Funct. Anal. 17 (2008), no. 5, 1343--1415

\bibitem{BCCT10}
\bysame,
{\em Finite bounds for H\"older-Brascamp-Lieb} multilinear inequalities,
Mathematical Research Letters 17(4): 647--666, 2010

\bibitem{carlenliebloss}
E.~Carlen, E.~Lieb, and M.~Loss,
{\em A sharp analog of Young's inequality on $S^N$ and related entropy inequalities}, 
J. Geom. Anal. 14 (2004), no. 3, 487--520

\bibitem{CDKSY}
M.~Christ, J.~Demmel, N.~Knight, T.~Scanlon, K.~Yelick,
{\em Communication lower bvounds and optimal algorithms for programs that reference arrays --- Part 1},
preprint

\bibitem{liebgaussian}
E.~Lieb,
{\em Gaussian kernels have only Gaussian maximizers}, 
Invent. Math. 102 (1990), no. 1, 179--208

\end{thebibliography}
\end{document}